\documentclass[12pt,reqno]{amsart} 
\usepackage{amssymb,amsthm,amsmath}
\usepackage{hyperref}
\usepackage{float}

\theoremstyle{plain}
\newtheorem{thm}{Theorem}
\newtheorem{lemma}{Lemma}
\newtheorem{prop}{Proposition}
\theoremstyle{definition}
\newtheorem*{problem*}{Problem}
\newtheorem{remark}{Remark}
\newtheorem{example}{Example}
\newtheorem*{notation*}{Notation}
\newtheorem*{observ*}{Observation}

\newcommand{\Z}{\mathbb{Z}}
\newcommand{\be}{\begin{enumerate}}
\newcommand{\ee}{\end{enumerate}}
\DeclareMathOperator{\Aut}{Aut}
\DeclareMathOperator{\im}{im}
\newcommand{\semidirect}{\rtimes} 
\renewcommand{\c}[1]{\langle #1 \rangle} 

\title[Cyclic Subgroups of Finite Groups]{Finite groups with a prescribed number of cyclic subgroups}
\author{Richard Belshoff}
\address{Department of Mathematics, Missouri State University, 
	Springfield, MO 65897}
\email{RBelshoff@MissouriState.edu}
\author{Joe Dillstrom}
\address{Department of Mathematics, Missouri State University, 
	Springfield, MO 65897}
\author{Les Reid}
\address{Department of Mathematics, Missouri State University, 
	Springfield, MO 65897}
\email{LesReid@MissouriState.edu}

\date{\today}

\begin{document}

\begin{abstract}
Marius T\u{a}rn\u{a}uceanu described the finite groups $G$ having $|G|-1$
cyclic subgroups.
We describe the finite groups $G$ having $|G|-\Delta$ cyclic subgroups
for $\Delta=2, 3, 4$ and $5$.
\end{abstract}

\maketitle


\section{Introduction}
Throughout this paper, a cyclic group of order $n$ is denoted $C_n$, and 
the dihedral group of order $2n$ is denoted $D_{2n}$.
When we write $H\semidirect K$ we include the possibility that
the semidirect product is a direct product.
In \cite{tuarn},
Marius T\u{a}rn\u{a}uceanu proved the following theorem.
\begin{thm}[T\u{a}rn\u{a}uceanu] A finite group $G$ has $|G|-1$ 
cyclic subgroups if and only if $G$ is isomorphic to 
one of the following groups:
	$C_3$, $C_4$, $S_3$, or $D_8$.
\end{thm}

Let $G$ be a finite group, and let $C(G)$ denote the poset of 
cyclic subgroups of $G$.
We define  $\Delta(G)$, or just $\Delta$ if the group is understood, to be 
the difference between the order of $G$ and the number of cyclic subgroups of $G$:
$$\Delta(G)=|G|-|C(G)|.$$  

It is known that $\Delta=0$ if and only if 
$G$ is an elementary abelian $2$-group, i.e.,
$G=C_2^n$ for some $n$.

In \cite{tuarn} where T\u{a}rn\u{a}uceanu described all groups
satisfying $\Delta=1$, he also 
posed the following natural open problem.

\begin{problem*} Describe the finite groups $G$ having $|G|-\Delta$
cyclic subgroups, where $2\le \Delta\le |G|-1$.
\end{problem*}

In this paper we solve this problem for $\Delta=2, 3, 4$ and $5$. 
More specifically we prove the following theorems. 

\begin{thm} A finite group $G$ has exactly $|G|-2$ cyclic subgroups 
	if and only if $G$ is isomorphic to one of the following groups:
	$C_4\times C_2$, $D_8\times C_2$, $C_6$, or $D_{12}$.
\end{thm}

\begin{thm} A finite group $G$ has exactly $|G|-3$ cyclic subgroups 
	if and only if $G$ is isomorphic to one of the following groups:
	$Q_8$, $C_5$, or $D_{10}$.
\end{thm}

\begin{thm} A finite group $G$ has exactly $|G|-4$ cyclic subgroups 
	if and only if $G$ is isomorphic to one of the following eleven groups:
	$C_4\times C_2\times C_2$,
	$C_2\times C_2\times D_8$,
	$(C_2\times C_2)\semidirect C_4$,
	$Q_8\semidirect C_2$,
	$C_3\times C_3$,
	$(C_3\times C_3)\semidirect C_2$,
	$A_4$,
	$C_6\times C_2$,
	$C_2\times C_2\times S_3$,
	$C_8$, or
	$D_{16}$.
\end{thm}

\begin{thm} A finite group $G$ has exactly $|G|-5$ cyclic subgroups 
	if and only if $G$ is isomorphic to one of the following groups:
        $C_7$, $D_{14}$, or $C_3\semidirect C_4$.
\end{thm}

\section{Preliminaries}
Let $G$ be a finite group.  Denote by $C(G)$ the poset 
of cyclic subgroups of $G$,
and let $\Delta(G)=|G|-|C(G)|$. 
For any positive integer $d$, let
$n_d= |\{H\in C(G)\ :\ |H|=d\}|$ be the number of cyclic subgroups 
of order $d$. 
Since every element of $G$ generates a cyclic subgroup and $\phi(d)$ is
the number of generators of a cyclic group of order $d$, it follows that
$$\sum_{d \ge 1} n_d\phi(d)=|G|.$$
Because $|C(G)|=\displaystyle\sum_{d \ge 1} n_d$, we have
\begin{equation} 
\label{eqn:tuarn}
	\sum_{d\ge 1} n_d(\phi(d)-1) = \Delta(G).
	\tag{$\star$}
\end{equation}

We will see below that the sum (\ref{eqn:tuarn}) gives information about 
the possible cyclic subgroups of $G$. 
Because of well-known properties of the  Euler $\phi$-function 
(namely, that $\phi(d)=1$ if and only if $d=1$ or $d=2$, and 
$\phi(d)$ is even for $d>2$),
we have the following.
\begin{remark} 
\label{remark:1}
For any group $G$ and any difference $\Delta=|G|-|C(G)|$, 
\be
\item $G$ may contain any number of cyclic subgroups of order $2$, and
\item if $n_d(\phi(d)-1)$ is a term in the sum (\ref{eqn:tuarn}), 
	then $\phi(d)-1$ must be odd for all $d>2$.
\ee
\end{remark}
\begin{notation*} We write $\sigma(G)=(m_1,m_2,\ldots, m_t)$,
	$m_i>2$, 
	to mean that 
	$G$ has one cyclic subgroup of order $m_1$, 
	another of order $m_2$, etc.,
	any number of cyclic subgroups of order $2$,
	and no other non-trivial cyclic subgroups.
	For example, $\sigma(Q_8)=(4,4,4)$
	and $\sigma(S_4)=(3,3,3,3,4,4,4)$.
\end{notation*}
\begin{example}
To see how the sum in (\ref{eqn:tuarn}) gives us information about 
the possible cyclic subgroups, consider $\Delta(G)=4$.
The sum in (\ref{eqn:tuarn}) is 
one of the partitions of $4$. That is, the 
left side of equation (\ref{eqn:tuarn}) is either
$$\mbox{(a) $4$,\quad (b) $3+1$,\quad (c) $2+2$,\quad (d) $2+1+1$,\quad or\quad 
(e) $1+1+1+1$.}$$
In case (a), there is a unique divisor $d$ of $|G|$ such that
$n_d(\phi(d)-1)=4\cdot 1$, or $2\cdot 2$, or $1\cdot 4$.
(When we say $n_d(\phi(d)-1)=4\cdot 1$, for example, we really mean
$(n_d, \phi(d)-1)=(4, 1)$, but we will continue to use the more
succinct notation.)
For all other divisors $i$ of $|G|$, we have $n_i(\phi(i)-1)=0$.
So either $n_i=0$, i.e., there are no cyclic subgroups of order $i$,
or $\phi(i)=1$, i.e., the divisor $i$ is $1$ or $2$.
By Remark \ref{remark:1}, we must have $n_d=4$ and 
$\phi(d)-1=1$ (i.e., $d\in\{3,4,6\}$).
So in this case $G$ has either exactly four cyclic subgroups of order $3$,
or exactly four cyclic subgroups of order $4$, or 
four cyclic subgroups of order $6$,
any number of cyclic subgroups of order $2$, and no other 
non-trivial
cyclic subgroups.
Using our notation, either $\sigma(G)=(3,3,3,3)$, 
or $\sigma(G)=(4,4,4,4)$, or $\sigma(G)=(6,6,6,6)$.
Note that if $\sigma(G)=(6,6,6,6)$ then $G$ also has a cyclic subgroup of order $3$, 
so $\sigma(G)=(6,6,6,6)$ is impossible. 
We will codify this observation in Proposition \ref{no_cn}.
\end{example}

In the following tables we list all the possibilities for $\sigma(G)$ for 
$1\le \Delta \le 5$.  
To construct these tables we use the following facts about the Euler $\phi$-function
$$\phi(d)=2 \iff  d\in\{3, 4, 6\},$$
$$\phi(d)=4 \iff  d\in\{5, 8, 10, 12\},\mbox{ and}$$
$$\phi(d)=6 \iff  d\in\{7, 9, 14, 18\}.$$
Just as we excluded $\sigma(G)=(6,6,6,6)$ in the example, we will 
exclude  many of the possibilities for $\sigma(G)$ in these tables
in the forthcoming sections.

\begin{table}[ht]
	\caption{\label{table:1} Table for $\Delta(G)=1$}
	\centering
\begin{tabular}{|c|c|}
\hline
Partition & $\sigma(G)$ \\
\hline
1 & (3), (4), (6) \\
\hline
\end{tabular}
\end{table}

\begin{table}[ht] 
	\caption{\label{table:2} Table for $\Delta(G)=2$}
	\centering
\begin{tabular}{|c|c|}
\hline
Partition & $\sigma(G)$ \\
\hline
2 & (3,3), (4,4), (6,6) \\
\hline
1+1 & (3,4), (3,6), (4,6) \\
\hline
\end{tabular}
\end{table}

\begin{table}[ht]
	\caption{\label{table:3}Table for $\Delta(G)=3$} 
	\centering
\begin{tabular}{|c|c|c|}
\hline
Partition & & $\sigma(G)$ \\
\hline
3 & $3\cdot 1$ & $(3,3,3)$, $(4,4,4)$, $(6,6,6)$ \\ \cline{2-3}
  & $1\cdot 3$ &  $(5)$, $(8)$, $(10)$, $(12)$ \\
\hline
2+1 & & $(3,3,4)$, $(3,3,6)$, $(3,4,4)$, $(4,4,6)$, $(3,6,6)$, $(4,6,6)$ \\
\hline
1+1+1 & & $(3,4,6)$ \\
\hline
\end{tabular}
\end{table}

\begin{table}[ht]
	\caption{\label{table:4}Table for $\Delta(G)=4$} 
	\centering
\begin{tabular}{|c|c|c|}
\hline
Partition & & $\sigma(G)$ \\
\hline
4 & & (3,3,3,3), (4,4,4,4), (6,6,6,6) \\
\hline
    & $3\cdot 1 + 1$ & (3,3,3,4), (3,3,3,6), (3,4,4,4), 
    (4,4,4,6), (3,6,6,6), (4,6,6,6)\\ \cline{2-3}
3+1 & $1\cdot 3 + 1$ & (3,5), (4,5), (5,6), (3,8), (4,8), (6,8), (3,10),\\
    & & (4,10), (6,10), (3,12), (4,12), (6,12) \\
\hline
2+2 & & (3,3,4,4), (3,3,6,6), (4,4,6,6) \\
\hline
2+1+1 & & (3,3,4,6), (3,4,4,6), (3,4,6,6) \\
\hline
1+1+1+1 & & none \\
\hline
\end{tabular}
\end{table}

\begin{table}[ht]
	\caption{\label{table:5}Table for $\Delta(G)=5$} 
	\centering
\begin{tabular}{|c|c|c|}
\hline
Partition & & $\sigma(G)$ \\
\hline
5 & $1\cdot 5$ & (7), (9), (14), (18) \\ \cline{2-3}
  & $5\cdot 1$ & (3,3,3,3,3), (4,4,4,4,4), (6,6,6,6,6) \\
\hline
4+1 &  & (3,3,3,3,4), (3,3,3,3,6), (3,4,4,4,4),  \\
  & & (3,6,6,6,6), (4,4,4,4,6), (4,6,6,6,6) \\
\hline
    & $3\cdot 1 + 2$ & (3,3,3,4,4), (3,3,3,6,6), (3,3,4,4,4), \\
  & & (4,4,4,6,6), (3,3,6,6,6), (4,4,6,6,6) \\ \cline{2-3}
3+2 & $1\cdot 3 + 2$ & (3,3,5), (4,4,5), (5,6,6), (3,3,8), \\
 & & (4,4,8), (6,6,8), (3,3,10), (4,4,10), \\
 & & (6,6,10), (3,3,12), (4,4,12), (6,6,12) \\
\hline
      & $3\cdot 1+1+1$ & (3,3,3,4,6), (3,4,4,4,6), (3,4,6,6,6) \\ \cline{2-3}
3+1+1 & $1\cdot 3 +1+1$ & (3,4,5), (3,5,6), (4,5,6), (3,4,8), \\
 & & (3,6,8), (4,6,8), (3,4,10), (3,6,10),\\
 & &   (4,6,10) (3,4,12), (3,6,12), (4,6,12) \\
 \hline
 2+2+1 & & (3,3,4,4,6), (3,3,4,6,6), (3,4,4,6,6) \\
\hline
2+1+1+1 & & none \\
\hline
1+1+1+1+1& & none \\
\hline
\end{tabular}
\end{table}

\subsection{Explanation of Exclusion Tables}
Continuing our example with $\Delta=4$, we see in Table \ref{table:4}
the 27 possible values for $\sigma(G)$.  We exclude 23 of these 
possibilities in Table \ref{table4.1} which we call the ``Exclusion Table
for $\Delta(G)=4$". The remaining 4 possible values for $\sigma(G)$ are
then listed in Table \ref{table4.2} which we call the
``Revised Table for $\Delta(G)=4$".

We will invoke the following three propositions frequently in the 
exclusion tables, referring to the first as ``No $C_n$", 
the second as ``Sylow", and the third as ``Proposition \ref{lemma:ab}".

\begin{prop}\label{no_cn}(``No $C_n$") If $m$ is an entry in $\sigma(G)$ and $n|m$ ($n>2$), then $n$ must also be an entry in $\sigma(G)$.
\end{prop}
\begin{proof} A cyclic subgroup of order $m$ must contain one of order $n$ if $n|m$.
\end{proof}

\begin{prop}\label{sylow}(``Sylow") If $p$ is a prime, $p$ divides $|G|$, and $n_p$ denotes the number of times $p$ occurs in $\sigma(G)$, then $n_p\equiv 1\bmod p$.
\end{prop}
\begin{proof} This follows immediately from Frobenius' generalization 
of the third part of Sylow's Theorem. 
(\cite{andreev}, \cite{brown}, or \cite{frobenius}.)
\end{proof}

\begin{prop}\label{lemma:ab} 
	Suppose a group $G$ has a unique cyclic subgroup $H$ of order $a$
	and a unique cyclic subgroup $K$ of order $b$. If $\gcd(a,b)=1$,
	then $G$ contains an element of order $ab$.
\end{prop}
\begin{proof} Since $H=\langle x\rangle$ and $K=\langle y\rangle$
	are unique, they are normal. The condition $\gcd(a,b)=1$ 
	implies $H\cap K=\{1\}$.  It follows that $xy=yx$
	and hence the element $xy$ has order $ab$.
\end{proof}

For example, in Table~\ref{table4.1} these three propositions eliminate all but two 
of the entries (which must be dealt with using auxiliary lemmas).

\subsection{More Propositions}
\begin{prop}\label{from_H} If $\sigma(G)=(d_1,\ldots,d_r)$ with 
$x_1,\ldots,x_r$ being generators for the cyclic subgroups 
of order $d_1,\ldots,d_r$ respectively and 
$H=\langle x_1,\ldots x_r\rangle$, then $G\cong H$ or $G\cong H\semidirect C_2$. 
The latter case can only occur when $H$ is abelian and 
if $t$ is the generator of $C_2$ then $tht^{-1}=h^{-1}$ for all $h\in H$.
\end{prop}
\begin{proof}
Since $H$ has cyclic subgroups of orders $d_1,\ldots,d_r$ and is a subgroup 
of $G$ which itself has no other non-trivial cyclic subgroups of those orders (except possibly subgroups of order 2), we must have $\sigma(H)=(d_1,\ldots,d_r)$. Therefore $G=H$ is one possibility. Suppose $H<G$, then there must be a $t\in G-H$ which is necessarily of order 2. Now $H\unlhd G$, since the cyclic subgroups of a given order must be conjugated to one another and they generate $H$. 
We will show that $\langle H,t\rangle\cong H\semidirect C_2$. 
For any $h\in H$, let $\alpha=tht^{-1}$. 
Since $ht\notin H$, it must have order 2, so 
$1=htht=h\alpha t^2=h\alpha$. Thus we must have 
$tht^{-1}=h^{-1}$ and every element in $\langle H,t\rangle - H$ 
must have order 2.
Therefore $\sigma(\langle H,t\rangle)=(d_1,\ldots,d_r)$ and 
$G=\langle H,t\rangle\cong H\semidirect C_2$ is another possibility. 
Since conjugation by $t$ is an automorphism of $H$, this forces 
$H$ to be abelian. 
Finally, suppose there is an $s\notin \langle H,t\rangle$. 
By the previous arguments $shs^{-1}=h^{-1}$ for all $h\in H$. 
But then $tsh(ts)^{-1} = h$ for all $h\in H$. 
In particular, if $|x_i|=d_i>2$, then $tsx_i\notin H$, 
but $|tsx_i|>2$ giving a contradiction.
\end{proof}

\begin{prop}\label{a_2a} \mbox{}
\begin{enumerate}
\item[(i)] We have $\sigma(G)=(a)$ if and only if $a=4$ 
or $a$ is an odd prime. In that case $G\cong C_a$ or $D_{2a}$ 
and $\Delta(G)=a-2$ in either case.
\item[(ii)] We have $\sigma(G)=(a,2a)$ if and only if $a=4$ or $a$ is an odd prime. In that case $G\cong C_{2a}$ or $D_{4a}$ and $\Delta(G)=2a-4$ in either case.
\end{enumerate}
\end{prop}
\begin{proof}
$(i)$ By Proposition \ref{no_cn} if $a$ had a proper odd divisor $d$, 
$d$ must appear as an entry in $\sigma(G)$, which it does not. 
Therefore $a$ must be an odd prime or a power of 2. If $a$ is a power of 2 
greater than 8, then 4 would also have to appear in $\sigma(G)$, 
which it does not. If $\sigma(G)=(a)$, when $a=4$ or an odd prime, 
we let $x$ be an element of order $a$ and apply Proposition \ref{from_H} 
to conclude that we must have $G\cong C_a$ or $D_{2a}$. 
One easily checks that that these groups satisfy $\sigma(G)=(a)$ 
and $\Delta(G)=a-2$.\\
\\
\noindent
$(ii)$ An argument analogous to that in $(i)$ shows that $a$ must be 4 or an odd prime. Suppose that $\sigma(G)=(a,2a)$. If $x$ is an element of order $2a$, then $x^2$ has order $a$. In the context of Proposition \ref{from_H}, $H=\langle x,x^2\rangle=\langle x\rangle$ and the only possibilities for $G$ are $C_{2a}$ and $D_{4a}$. One readily verifies that in these cases, $\sigma(G)=(a,2a)$ and $\Delta(G)=2a-4$.
\end{proof}

\section{Difference 1}
In this section we reprove T\u{a}rn\u{a}uceanu's theorem \cite{tuarn}
in order to illustrate the methods we will use in the later sections.

\setcounter{thm}{0}
\begin{thm} A finite group $G$ has $|G|-1$ cyclic subgroups 
	if and only if $G$ is isomorphic to one of the following groups:
	$C_3$, $C_4$, $S_3$, or $D_8$.
\end{thm}

\begin{proof} One easily checks that if $G$ is one of the listed
	groups then $G$ has $|G|-1$ cyclic subgroups.  
	Conversely, if $G$ has exactly $|G|-1$ cyclic subgroups,
	then from Table~\ref{table:1} either
	$\sigma(G)=(3), (4)$ or $(6)$.
We immediately exclude $\sigma(G)=(6)$ by Proposition \ref{no_cn}. Applying Proposition 
\ref{a_2a}(i), 
we conclude that $\sigma(G)=3$ implies $G\cong C_3$ or $D_6$ and $\sigma(G)=4$ implies $G\cong C_4$ or $D_8$. 
\end{proof}

\begin{table}[ht]
	\caption{\label{table1.1} Exclusion table for $\Delta(G)=1$}
	\centering
\begin{tabular}{|c|l|} 
	\hline
	$\sigma(G)$ & Excluded by \\
	\hline
	(6) & No $C_3$ \\
	\hline
\end{tabular}
\end{table}

\begin{table}[ht]
	\caption{\label{table1.2} Revised table for $\Delta(G)=1$}
	\centering
\begin{tabular}{|c|c|c|}
\hline
Partition & $\sigma(G)$ & Groups \\
\hline
1 & (3) & $C_3$, $S_3$ \\
\hline
  & (4) & $C_4$, $D_{8}$ \\
\hline
\end{tabular}
\end{table}

\section{Difference 2}
In this section we prove the following.
\setcounter{thm}{1}
\begin{thm}\label{thm:diff2}
       	A finite group $G$ has exactly $|G|-2$ cyclic subgroups 
	if and only if $G$ is isomorphic to one of the following groups:
	$C_4\times C_2$, $D_8\times C_2$, $C_6$, or $D_{12}$.
\end{thm}

The following lemma will be used below, and later in the
Difference~4 section of the paper.

\begin{lemma}\label{lemma:Z}
If $G$ contains fewer than 6 cyclic subgroups of order 4 
and $s$ and $t$ are distinct elements of order 4, 
and $tst^{-1}\in\c{s}$, then $s^2=t^2$.
\end{lemma}

\begin{proof} If, on the contrary, $\langle s\rangle\cap \langle t\rangle=\{1\}$, then 
$$\langle s,t\rangle\cong C_4\semidirect C_4\cong\langle s,t|s^4=t^4=1,tst^{-1}=s^{\pm 1}\rangle.$$
In either case, one readily verifies that the elements other than 
$1,s^2,t^2,s^2t^2$ all have order 4.  These 12 elements yield 6 cyclic subgroups 
of order 4, a contradiction. 
\end{proof}

We now begin the proof of Theorem \ref{thm:diff2}.
\begin{proof}[Proof of Theorem \ref{thm:diff2}]
If $G$ is any one of the groups 
$C_4\times C_2$, $D_8\times C_2$, $C_6$, or $D_{12}$,
then $G$ has $|G|-2$ cyclic subgroups.

Conversely, if $G$ is a finite group with exactly $|G|-2$ cyclic subgroups,
then from Table \ref{table:2} we know the possible values of $\sigma(G)$.
If $\sigma(G)=(3,3)$, then $G$ has two Sylow $3$-subgroups, but this is 
impossible 
by Proposition \ref{sylow} (``Sylow'').
We exclude $\sigma(G)=(6,6)$ and $\sigma(G)=(4,6)$
by Proposition~\ref{no_cn} (``No $C_n$'').
The case $\sigma(G)=(3,4)$ is excluded by Proposition \ref{lemma:ab}.
We summarize these arguments in table form in Table \ref{table2.1}.

\begin{table}[ht]
	\caption{\label{table2.1} Exclusion table for $\Delta(G)=2$}
	\centering
\begin{tabular}{|c|l|} 
	\hline
	$\sigma(G)$ & Excluded by \\
	\hline
	(3,3) & Sylow \\
	(6,6) & No $C_3$ \\
	(3,4) & Proposition \ref{lemma:ab} \\
	(4,6) & No $C_3$ \\
	\hline
\end{tabular}
\end{table}

The two Propositions below show that 
we have the revised Table~\ref{table2.2} and complete the proof.
\end{proof}

\begin{prop} If $\sigma(G)=(4,4)$ then  $G=C_4\times C_2$ or $G=D_8\times C_2$.  
\end{prop} 
\begin{proof} 
Let $H_1=\c{x}$ and $H_2=\c{y}$ be the two cyclic subgroups of order $4$
and let $H=\c{x,y}$.  We must have $yxy^{-1}\in H_1$ (since
$yxy^{-1}\in H_2$ gives $x\in H_2$), so by Lemma \ref{lemma:Z}, 
$x^2=y^2$ and hence $H\cong C_4\times C_2$ (if $yxy^{-1}=x$), or
$H\cong Q_8$ (if $yxy^{-1}=x^{-1}$). Since $Q_8$ has three cyclic subgroups
of order $4$, we must have $H\cong C_4\times C_2$.  By Proposition 
\ref{from_H}, $G\cong C_4\times C_2$ or $G\cong (C_4\times C_2)\semidirect C_2$,
where the action of $C_2$ in the latter group is to send an element of $C_4\times C_2$
to its inverse.  It is easy to see that in this case
$(C_4\times C_2)\semidirect C_2 \cong D_8\times C_2$ and the result follows.
\end{proof}       

\begin{prop} If $\sigma(G)=(3,6)$ then $G=C_6$ or $G=D_{12}$.  \end{prop}
\begin{proof} 
This follows from Proposition \ref{a_2a} (ii).	
\end{proof}

\begin{table}[ht]
	\caption{\label{table2.2} Revised table for $\Delta(G)=2$}
	\centering
\begin{tabular}{|c|c|c|}
\hline
Partition & $\sigma(G)$ & Groups \\
\hline
2 & (4,4) & $C_4\times C_2$, $D_8\times C_2$ \\
\hline
1+1 & (3,6) & $C_6$, $D_{12}$ \\
\hline
\end{tabular}
\end{table}

Before moving on to the next section we make the following observation, 
which is not difficult to prove.  
\begin{observ*}
If $G$ is a group with $\Delta(G)=d$, then $\Delta(G\times C_2)= 2d$.
\end{observ*}
Note that every group with difference $2$ is of the form $G\times C_2$
where $\Delta(G)=1$.  But we will see later that not every group 
with difference $4$ is of the form $G\times C_2$ where $\Delta(G)=3$

\section{Difference 3}
Our goal in this section is to prove the following.

\setcounter{thm}{2}
\begin{thm}\label{thm:diff3}
	A finite group $G$ has exactly $|G|-3$ cyclic subgroups 
	if and only if $G$ is isomorphic to one of the following groups:
	$Q_8$, $C_5$, or $D_{10}$.
\end{thm}

We will need the following lemmas.

\begin{lemma} \label{lemma:les1}
If a group $G$ has a cyclic subgroup of order $4$ and a unique 
cyclic subgroup of order $3$, then $G$ has an element of order $6$.
\end{lemma}
\begin{proof} Suppose 
	$\langle x\rangle$ is a cyclic subgroup of order $4$ and 
	$\langle y\rangle$ is the unique cyclic subgroup of order $3$.
	Then $\langle y\rangle$ is normal, so 
	$xyx^{-1}=y$ or $xyx^{-1}=y^{-1}$. In either case
	$x^2yx^{-2} = y$ and $x^2y=yx^2$. Then the element
	$x^2y$ has order 6, which completes the proof.
\end{proof}

\begin{lemma} \label{lemma:les2}
If $G$ has exactly two cyclic subgroups of order $4$ and a cyclic subgroup 
of order $3$, then $G$ has an element of order $12$.
\end{lemma}
\begin{proof}
Let $H_1=\langle x\rangle$, $H_2=\langle y\rangle$, and  $\langle z\rangle$ 
be the cyclic subgroups of orders $4, 4$ and $3$ respectively.
The group $G$ acts on $\{H_1, H_2\}$ by conjugation, 
and under the induced homomorphism $G\to S_2$ we have $z\mapsto (1)$. 
In particular
$zxz^{-1}=x^{\pm 1}$, and then $z^2xz^{-2}=x$.  So $z^2x = xz^2$ and 
the element $xz^2$ has order $12$.
\end{proof}

\begin{lemma} \label{lemma:366}
	There is no group with exactly two cyclic subgroups of order~6
	and a unique subgroup of order~3.
\end{lemma}
\begin{proof} 
Let $H_1=\c{x}$ and $H_2=\c{y}$ be the two cyclic subgroups of order~6.
In order for there to be a unique subgroup of order~3, it must be the
case that $|H_1\cap H_2|=3$. In that case, without loss of generality,
$x^2=y^2$. 
	The element
	$xyx^{-1}$ has order $6$ and cannot be in $\langle x\rangle$, 
	so either $xyx^{-1}=y$ or $xyx^{-1}=y^{-1}$.  In both
	cases one can check that the subgroup $\langle xy \rangle$
	is cyclic of order $6$ and is distinct from 
	$\langle x\rangle$ and $\langle y\rangle$, a contradiction.
\end{proof}

\begin{table}[ht]
	\caption{\label{table3.1} Exclusion table for $\Delta(G)=3$}
	\centering
\begin{tabular}{|c|l|} 
	\hline
	$\sigma(G)$ & Excluded by \\
	\hline
	(3,3,3) & Sylow \\
	(6,6,6) & No $C_3$ \\ 
	(8) & No $C_4$ \\
	(10) & No $C_5$ \\
	(12) & No $C_3$\\
	(3,3,4) & Sylow \\
	(3,3,6) & Sylow \\
	(3,4,4) & Lemma \ref{lemma:les1} \\
	(4,4,6) & No $C_3$ \\
	(3,6,6) & Lemma \ref{lemma:366} \\
	(4,6,6) & No $C_3$ \\
	(3,4,6) & Proposition \ref{lemma:ab} \\
	\hline
\end{tabular}
\end{table}

\setcounter{thm}{5}
\begin{thm}\label{herzog}If $G$ is a 2-group with an odd number of cyclic subgroups of order 4, then $G$ must be cyclic, dihedral, generalized quaternion, or quasidihedral.
\end{thm}
\begin{proof} 
Let $i_n(G)$ denote the number of elements of order $n$ in $G$. If the number of cyclic subgroups of order 4 is $2k+1$, then $i_4(G)=2(2k+1)$. Since $|G|\ge 4$, the number of solutions to $x^4=1$ in $G$ must be a multiple of 4 by Frobenius' 
Theorem \cite[Theorem 5]{IR}. This number is $i_4(G)+i_2(G)+i_1(G)=2(2k+1)+i_2(G)+1$, so $i_2(G)\equiv 1\bmod 4$. It is known that in this case, $G$ must be cyclic, dihedral, generalized quaternion, or quasidihedral (\cite[Theorem 4.9]{isaacs} gives a character-theoretical proof and \cite{herzog} gives an elementary one).  
\end{proof}

\begin{prop}\label{odd_4s} If $\sigma(G)=(4,\ldots,4)$ with $m$ 4's and $m$ is odd, 
	then $m=1$ and $G\cong C_4$ or $D_8$ (and $\Delta(G)=1$ in both cases) 
	or $m=3$ and $G\cong Q_8$ (and $\Delta(G)=3$).
\end{prop}
\begin{proof} 
Now $G$ must be a 2-group (if not, $G$ would have elements of non-trivial 
odd order) so by Theorem \ref{herzog}, $G$ must be cyclic, dihedral, 
generalized quaternion, or quasidihedral.  If $|G|\ge 16$, then each of these 
types of group contain an element of order 8, which is impossible. 
Therefore $|G|=4$ or $|G|=8$. 
In the first case, $G\cong C_4$ with $\sigma(G)=(4)$ and $\Delta(G)=1$. 
In the second, $G$ must either be isomorphic to $C_8$, $D_8$, or $Q_8$.
But $G$ has no element of order $8$, so either
$G\cong D_8$, or $Q_8$ 
(with $\sigma(D_8)=(4), \Delta(D_8)=1, \sigma(Q_8)=(4,4,4)$, and $\Delta(Q_8)=3$).   
\end{proof}

\begin{prop}\label{d3:prop3}
       	If $\sigma(G)=(5)$, then $G=C_5$ or $G=D_{10}$.
\end{prop}
\begin{proof}
This follows from Proposition \ref{a_2a} (i).
\end{proof}

\begin{prop}\label{prop:444}
If $\sigma(G)=(4,4,4)$, then $G\cong Q_8$.
\end{prop}

\begin{proof} 
This follows from Proposition \ref{odd_4s}.
\end{proof}

\begin{remark}
	We remark that Proposition \ref{odd_4s} gives another proof of the $\sigma(G)=(4)$
	case in the Difference~$1$ section.
\end{remark}

\begin{table}[ht]
	\caption{\label{table3.2} Revised table for $\Delta(G)=3$}
	\centering
\begin{tabular}{|c|c|c|c|}
\hline
Partition & & $\sigma(G)$ & Groups \\
\hline
3 & $3\cdot 1$ & $(4,4,4)$ & $Q_8$\\ \cline{2-4}
  & $1\cdot 3$ &  $(5)$  & $C_5$, $D_{10}$\\
\hline
\end{tabular}
\end{table}

We are now ready to prove the main theorem of this section.

\begin{proof}[Proof of Theorem \ref{thm:diff3}]
If $\Delta(G)=3$, then by Table~\ref{table3.2} either $\sigma(G)=(5)$ or $\sigma(G)=(4,4,4)$.
If $\sigma(G)=(5)$, then $G\cong C_5$ or $G\cong D_{10}$ by Proposition~\ref{d3:prop3}. 
If $\sigma(G)=(4,4,4)$, then $G\cong Q_8$ by Proposition~\ref{prop:444}. 

Conversely, one checks that the groups $C_5$, $D_{10}$ and $Q_{8}$ all have
$\Delta=3$.
\end{proof}

\section{Difference 4}
In this section we prove the following theorem.
\setcounter{thm}{3}
\begin{thm}\label{thm:diff4}
	A finite group $G$ has exactly $|G|-4$ cyclic subgroups 
	if and only if $G$ is isomorphic to one of the following eleven groups:
	$C_4\times C_2\times C_2$,\ 
	$C_2\times C_2\times D_8$,\ 
	$(C_2\times C_2)\semidirect C_4$,\ 
	$Q_8\semidirect C_2$,\ 
	$C_3\times C_3$,\ 
	$(C_3\times C_3)\semidirect C_2$,\ 
	$A_4$,\ 
	$C_6\times C_2$,\ 
	$C_2\times C_2\times S_3$,\ 
	$C_8$, or
	$D_{16}$.
\end{thm}

We list the exclusion table and the revised table for $\Delta=4$
in Tables~\ref{table4.1} and \ref{table4.2}.

\begin{table}[ht]
	\caption{\label{table4.1} Exclusion table for $\Delta(G)=4$}
	\centering
\begin{tabular}{|c|l|} 
	\hline
	$\sigma(G)$ & Excluded by \\
	\hline
	(6,6,6,6) & No $C_3$ \\
	(3,3,3,4) & Sylow \\
	(3,3,3,6) & Sylow \\
	(3,4,4,4) & Lemma \ref{lemma:les1} \\
	(4,4,4,6) & No $C_3$ \\
	(4,6,6,6) & No $C_3$ \\
	(3,5) & Proposition \ref{lemma:ab} \\
	(4,5) & Proposition \ref{lemma:ab} \\
	(5,6) & No $C_3$ \\
	(3,8) & No $C_4$ \\
	(6,8) & No $C_4$ \\
	(3,10) & No $C_5$ \\
	(4,10) & No $C_5$ \\
	(6,10) & No $C_3$ \\
	(3,12) & No $C_4$ \\
	(4,12) & No $C_3$ \\
	(6,12) & No $C_3$ \\
	(3,3,4,4) & Sylow \\
	(3,3,6,6) & Sylow \\
	(4,4,6,6) & No $C_3$ \\
	(3,3,4,6) & Sylow \\
	(3,4,4,6) & Lemma \ref{lemma:les2} \\
	(3,4,6,6) & Proposition \ref{lemma:ab} \\
	\hline
\end{tabular}
\end{table}

\begin{table}[ht]
	\caption{\label{table4.2} Revised table for $\Delta(G)=4$} 
	\centering
\begin{tabular}{|c|c|c|c|}
\hline
Partition & & $\sigma(G)$ & Groups \\
\hline
4 & $4\cdot 1$ & (3,3,3,3) & $C_3\times C_3$, $(C_3\times C_3)\semidirect C_2$, $A_4$\\ \cline{2-4}
  & $4\cdot 1$ & (4,4,4,4) & $C_4\times C_2\times C_2$, $C_2\times C_2\times D_8$, \\
  &  &  & $(C_2\times C_2)\semidirect C_4$, $Q_8\semidirect C_2$ \\  
\hline
    & $3\cdot 1 + 1$ & (3,6,6,6) & $C_6\times C_2$, $C_2\times C_2\times S_3$\\ \cline{2-4}
3+1 & $1\cdot 3 + 1$ & (4,8)  & $C_8$, $D_{16}$ \\
\hline
\end{tabular}
\end{table}

\begin{prop}\label{prop:3333} If $\sigma(G)=(3,3,3,3)$ then $G$ is 
	isomorphic to either 
	$C_3\times C_3$, $(C_3\times C_3)\semidirect C_2$, or $A_4$.
\end{prop}
\begin{proof}
Let $H_i=\langle x_i\rangle, i=1,2,3,4$, be the distinct cyclic subgroups 
of order 3 and let $H=\langle x_1,x_2,x_3,x_4\rangle$.  
We have $H$ acting on $\{H_1, H_2, H_3, H_4\}$ by 
conjugation giving a homomorphism $\phi:H\to S_4$ 
(where we identify $H_i$ with $i$ in $S_4$).
Each $x_i$ must map to either a 3-cycle or the identity in $S_4$, 
since each $x_i$ has order $3$. Also, $\phi(x_i)$ must fix $i$, 
since $x_i H_i x_i^{-1}=H_i$.\\
\\
\noindent
Claim: We have $x_ix_jx_i^{-1}\ne x_j^{-1}$.\\
\\
\noindent
Proof of claim: If $x_ix_jx_i^{-1}= x_j^{-1}$, then $(x_jx_i)^2=x_i^2\ne 1$, 
and $(x_jx_i)^3=x_j\ne 1$, hence $|x_jx_i|\ge 4$, a contradiction.\\
\\
\noindent
Returning to the proof, there are two cases: 
either $\phi(x_i)=1$ for some $i$, or $\phi$ sends all the $x_i$ 
to 3-cycles, without loss of generality 
$\phi(x_1)=(234),\phi(x_2)=(134),\phi(x_3)=(124)$, and $\phi(x_4)=(123)$. \\
\\
\noindent
In the first case, we may assume $\phi(x_1)=1$ without loss of generality. 
Then $x_1x_2x_1^{-1}\in\c{x_2}$ so either 
$x_1x_2x_1^{-1} = x_2$ or 
$x_1x_2x_1^{-1} = x_2^{-1}$.  By the claim, we must have the former, i.e.,
$x_1x_2=x_2x_1$, so $\langle x_1,x_2\rangle\cong C_3\times C_3$. 
This groups has four cyclic subgroups of order 4, so we must have 
$H=\langle x_1,x_2\rangle\cong C_3\times C_3$. 
By Proposition \ref{from_H}, $G\cong C_3\times C_3$ or  
$(C_3\times C_3)\semidirect C_2$, where the action of 
$C_2$ sends each element of $C_3\times C_3$ to its inverse.\\
\\
\noindent
In the second case we claim $G\cong A_4$. 
Since $\phi$ maps the $x_i$ to 3-cycles that generate $A_4$, 
$\im\phi=A_4$. We claim that $\ker \phi = 1$. 
Note that if $v\ne 1$ and $v\in\ker \phi$, then $|v|=2$, 
since none of the order~3 elements are in the kernel.
\\
\noindent
Let $K = \ker \phi$. 
Now $K \cap \langle x_1 \rangle = 1$, 
$\phi(\langle K,x_1 \rangle) = \langle (234) \rangle$, and from this it follows 
that if $\alpha \in \langle K,x_1 \rangle - \langle x_1 \rangle$, then $|\alpha| = 2$.\\
\\
\noindent 
Suppose that the claim is false, then there exists $u \in K, u \neq 1.$ 
Consider $ux_1 \in \langle K,x_1 \rangle$. Then $ux_1 \notin \langle x_1 \rangle$, 
for if it were then $u \in \langle x_1 \rangle$. Thus, $|ux_1| = 2$. Therefore
$$1 = (ux_1)^{2} = ux_1ux_1 = ux_1u^{-1}x_1^{-1}x_1^{2},$$
which implies that $x_1^{2} = x_1ux_1^{-1}u^{-1}$, where $x_1ux_1^{-1} \in K$ 
by the normality of $K$, hence $x_1^{2} \in K$, a contradiction.
Thus, $\ker \phi = 1$ and $H\cong A_4$. By Proposition \ref{from_H}, $G\cong H\cong A_4$.
\end{proof}

\begin{prop}\label{prop:4444} If $\sigma(G)=(4,4,4,4)$ then $G$ is isomorphic to 
$C_4\times C_2\times C_2$,
$D_8\times C_2\times C_2$,
$(C_2\times C_2)\semidirect C_4$, or
$Q_8\semidirect C_2$.
The last two groups are SmallGroup(16,3) and SmallGroup(16,13) respectively in GAP.
\end{prop}
\begin{proof}

Suppose the four cyclic subgroups of order~4 are $H_i=\langle x_i\rangle$ 
for $i=1, 2, 3, 4$. If $H=\langle x_1,x_2,x_3,x_4\rangle$ then $H$ acts on the set 
$\{H_1, H_2, H_3, H_4\}$ by conjugation, inducing a homomorphism
$\phi: H\rightarrow S_4$ 
(as usual, when representing a permutation of the 
$H_i$ as an element of $S_4$, we will replace $H_i$ by $i$).
Now the order of $\phi(x_i)$ must divide 4 and $\phi(x_i)$ must fix $i$. 
This eliminates 4-cycles or the product of two disjoint 2-cycles as possibilities. 
In particular, $\phi(x_1)=(1), (23),(24)$, or $(34)$ and, 
without loss of generality, we may assume $\phi(x_1)=(1)$ or $(34)$. 
Note that $x_1 H_2 x_1^{-1}=H_2$ in either case.
Similarly, $\phi(x_2)=(1), (13),(14)$, or $(34)$. However, in the second and third cases $\phi(x_1x_2)$ is a
3-cycle, which is impossible (as all the elements of $H$ have orders 1,2, or 4). Therefore $\phi(x_2)=(1)$ or $(34)$. Analogous arguments show 
$\phi(x_3)=(1)$ or $(12)$ and $\phi(x_4)=(1)$ or $(12)$.\\
\\
\noindent
Because $x_1 H_2 x_1^{-1}=H_2$, by Lemma \ref{lemma:Z}, $x_1^2=x_2^2$. 
If $x_1x_2x_1^{-1}=x_2$, then 
$K_1=\langle x_1,x_2\rangle\cong C_4\times C_2$, and if $x_1x_2x_1^{-1}=x_2^{-1}$, then $K_1\cong Q_8$. Similarly, 
$K_2=\langle x_3,x_4\rangle\cong C_4\times C_2$ or $Q_8$. In all of the cases above, $K_1,K_2\unlhd H$. \\
\\
\noindent
We claim that $|H|=16$. We have $H=K_1K_2$, so $|H|=|K_1||K_2|/|K_1\cap K_2|$. 
We cannot have $K_1=K_2$, since then $H=K_1=K_2$ and 
$C_4\times C_2$ has only two cyclic subgroups of order 4 and 
$Q_8$ has only three. If either $K_1\cong Q_8$ or $K_2\cong Q_8$ 
(without loss of generality, assume it's $K_1$), then there is a 
third cyclic subgroup of order 4 in $K_1$, which must be either 
$\langle x_3\rangle$ or $\langle x_4\rangle$. This gives us
$|K_1\cap K_2|=4$ and hence $|H|=16$ as desired. 
We may therefore assume that $K_1\cong K_2\cong C_4\times C_2$.\\
\\
\noindent
If $|K_1\cap K_2|=1$, then $H\cong K_1\times K_2\cong C_4^2\times C_2^2$ 
and $C_4^2$ itself has 6 cyclic subgroups of order 4, which is more than 
we are allowed.
Consequently,
$H/(K_1\cap K_2)\cong (K_1/(K_1\cap K_2))\times(K_2/(K_1\cap K_2))$ 
is abelian, being the product of two groups each of which is of 
order at most $4$.
\\
\noindent
If $|K_1\cap K_2|=2$, then $K_1\cap K_2$ is a subgroup of $K_1$ of order 2 and must therefore be generated by $x_1^2, x_1x_2$, or $x_1^{-1}x_2$.  The latter two cases are equivalent (we merely choose $x_1^{-1}$ to be the generator of $H_1$ and repeat the analogous arguments as those with $x_1$ as the generator).\\
\\
\noindent
Case 1: Suppose $K_1\cap K_2\le K_1$ is generated by $x_1^2=x_2^2$. 
Since the image of $x_1x_3x_1^{-1}x_3^{-1}$ in 
$H/(K_1\cap K_2)$ is trivial, we must have 
$x_1x_3x_1^{-1}x_3^{-1}=x_1^2$ or $1$. In the former case, 
$x_3x_1x_3^{-1}=x_1^{-1}$ and $\langle x_1,x_3\rangle\cong Q_8$. But here, 
there is a third cyclic subgroup of order 4 in $Q_8$, so either 
$x_2$ or $x_4$ is in $\langle x_1,x_3\rangle$. In the first subcase, $x_2$ does not commute with $x_1$ and in the second $x_4$ does not commute wth $x_3$, contradicting the fact that $\langle x_1,x_2\rangle$ and $\langle x_3,x_4\rangle$ are abelian. So the latter case holds and $x_1$ and $x_3$ must commute. Replacing $x_3$ by $x_4$ and $x_1$ by $x_2$ and repeating the argument above shows that $H$ must be abelian. Now $|H|=32$ and $H$ only has elements of order 1,2, or 4, so $H\cong C_4\times C_2^3$ or $C_4^2\times C_2$. In both cases, $H$ has too many cyclic subgroups of order 4 (8 in the first case and 12 in the second).  An analogous argument shows that $K_1\cap K_2\le K_2$ cannot be generated by $x_3^2=x_4^2$.\\  
\\
\noindent
Case 2: Suppose $K_1\cap K_2\le K_1$ is generated by $x_1x_2$ and $K_1\cap K_2\le K_2$ is generated by $x_3x_4$. In this case $H/(K_1\cap K_2)\cong C_4\times C_4$, which has 6 cyclic subgroups of order 4, giving a contradiction.\\
\\
\noindent
Therefore $|K_1\cap K_2|=4$ and $|H|=16$. Rather than proceeding 
further using first principles, we will invoke the classification of 
groups of order 16.  
By \cite[Table 1]{wild},
the only groups of order 16 with 4 cyclic subgroups of order 4 and no cyclic subgroups of larger order are $C_4\times C_2\times C_2$,
$$(C_2\times C_2)\semidirect C_4\cong\langle a,b,c|a^2=b^2=c^4=1,ab=ba,ca=ca,cbc^{-1}=ab\rangle,\mbox{ or}$$
$$Q_8\semidirect C_2\cong\langle a,b,c|a^4=c^2=1,a^2=b^2,bab^{-1}=a^{-1},ca=ca,cbc^{-1}=a^2b\rangle.$$
\\
\noindent
By Proposition \ref{from_H}, the only groups with $\sigma(G)=(4,4,4,4)$ are 
the groups above and $(C_4\times C_2\times C_2)\semidirect C_2$ 
where the action of $C_2$ on 
$C_4\times C_2\times C_2$ is $x\rightarrow x^{-1}$. 
It is straightforward to verify that this group is isomorphic to $D_8\times C_2\times C_2$.
\end{proof}

\begin{prop}\label{prop:3666}
       	If $\sigma(G)=(3,6,6,6)$, then 
	$G\cong C_6\times C_2$ 
	or $G\cong C_2\times C_2\times S_3$.
\end{prop}
\begin{proof}
Let $G$ be a finite group with exactly three cyclic subgroups of order 
$6$, exactly one cyclic subgroup of order $3$, and any number of 
cyclic subgroups of order $2$.

Let 
$H_1=\langle x\rangle$,
$H_2=\langle y\rangle$,
$H_3=\langle z\rangle$ be the three cyclic subgroups of order $6$ and let $H=\langle x,y,z\rangle$.
Then we may choose generators so that $x^2=y^2=z^2=w$, 
and $\langle w\rangle \cong C_3$.

We have $H$ acting on $\{H_1, H_2, H_3\}$ by conjugation, which gives a
homomorphism $\phi: H\to S_3$.  We have $\phi(x) = 1$ or $\phi(x)=(23)$; $\phi(y)=1$ or $\phi(y)=(13)$; and 
$\phi(z)=1$ or $\phi(z)=(12)$.
Note also that $\phi(w)=1$.

\smallskip

\noindent
We claim that no two of $x,y,z$ can map to a transposition. Without loss of generality, assume that $\phi(x)=(23)$ and $\phi(y)=(13)$. Then $\phi(xy)=(123)$, so $|xy|$ must be 3 or 6. 
But only $w^{\pm 1}$ has order~3, 
and $x^{\pm 1},y^{\pm 1},z^{\pm 1}$ are the only elements of order 6,
and these all map to the identity or a transposition. This yields a contradiction.

\smallskip

By the claim, we may assume, without loss of generality, that $\phi(x)=1$, so $x\langle y\rangle x^{-1}=\langle y\rangle$. 
Since $x^2=y^2$, 
$\langle x,y\rangle = \langle x^3,y\rangle$. 
We also have $\langle y\rangle \unlhd \langle x^3,y\rangle$ and 
$\langle x^3\rangle \cap \langle y\rangle =1$, 
so $\langle x,y\rangle\cong C_6\semidirect C_2$. 
This means $\langle x,y\rangle \cong C_6\times C_2$ or $D_{12}$. 
But $\langle x,y\rangle$ is supposed to have at least two cyclic subgroups 
of order 6 and $D_{12}$ only has one. Since $C_6\times C_2$ has three 
subgroups of order 6 and one of order 3, 
$H=\langle x,y\rangle\cong C_6\times C_2$. By Proposition \ref{from_H}, 
$G\cong C_6\times C_2$ or $G\cong (C_6\times C_2)\semidirect C_2$, 
where the action of $C_2$ sends each element of $C_6\times C_2$
to its inverse.
It is straightforward to verify that this latter group is isomorphic to 
$C_2\times C_2\times S_3$.\\
\\
\noindent
One readily verifies that $C_6\times C_2$ and $C_2\times C_2\times S_3$ both satisfy the conditions of the theorem.
\end{proof}

\begin{prop}\label{prop:48}
If $\sigma(G)=(4,8)$ then $G$ is isomorphic to either $C_8$ or $D_{16}$.
\end{prop}
\begin{proof}
	This follows from Proposition \ref{a_2a} (ii).
\end{proof}

We are now ready to prove Theorem~\ref{thm:diff4}.

\begin{proof}[Proof of Theorem \ref{thm:diff4}]
Our work above shows that if 
$\Delta(G)=4$, then $\sigma(G) = (3,3,3,3)$, $(4,4,4,4)$, $(3,6,6,6)$ or $(4,8)$.
If $\sigma(G)=(3,3,3,3)$ then $G$ is 
   isomorphic to either $C_3\times C_3$,
   $(C_3\times C_3)\semidirect C_2$, or $A_4$ by Proposition~\ref{prop:3333}.
If $\sigma(G)=(4,4,4,4)$ then $G$ is isomorphic to either 
$C_4\times C_2\times C_2$, $C_2\times C_2\times D_8$,
$(C_2\times C_2)\semidirect C_4$, or $Q_8\semidirect C_2$
   by Proposition~\ref{prop:4444}.
If $\sigma(G)=(3,6,6,6)$, then $G\cong C_6\times C_2$ 
	or $G\cong C_2\times C_2\times S_3$ by Proposition~\ref{prop:3666}.
If $\sigma(G)=(4,8)$ then $G\cong C_8$ or $G\cong D_{16}$ by Proposition~\ref{prop:48}. 

Conversely, one checks all of the eleven groups listed have $\Delta=4$.
\end{proof}

\section{Difference 5}
\setcounter{thm}{4}
\begin{thm}\label{thm:diff5}
	A finite group $G$ has exactly $|G|-5$ cyclic subgroups 
	if and only if $G$ is isomorphic to one of the following groups:
        $C_7$, $D_{14}$, or $C_3\semidirect C_4$.
\end{thm}

\begin{lemma} \label{lemma:r}
If a group $G$ has a unique cyclic subgroup of order $4$ and a 
subgroup of order $3$, then $G$ has an element of order $12$.
\end{lemma}
\begin{proof} Suppose 
	$\langle x\rangle$ is the unique cyclic subgroup of order $4$ and 
	$\langle y\rangle$ is a cyclic subgroup of order $3$.
	Then $\c{x}$ is normal in $\c{x, y}$, and 
	$\langle x\rangle \cap \langle y\rangle = \{1\}$, 
	so  we have a semidirect product
	$\langle x\rangle \semidirect\langle y\rangle$ for
	some homomorphism $\langle y\rangle \to \Aut(C_4)$.
	But $\Aut(C_4)\cong\Z_4^\times\cong C_2$ and $y$ has order $3$,
	so the homomorphism must be trivial.
	Hence 
	$\langle x\rangle \semidirect \langle y\rangle
	\cong C_4\times C_3  \cong  C_{12}$ and $G$ has an 
	element of order $12$.
\end{proof}

\begin{lemma}\label{lemma:36666}
There is no group $G$ such that $\sigma(G)=(3,6,6,6,6)$.
\end{lemma}
\begin{proof}
Let $H_1,H_2,H_3,H_4$ be the four cyclic subgroups of order 6 
with generators $x_1, x_2, x_3, x_4$ respectively. Since there is a unique 
subgroup of order 3, we must have $x_1^2=x_2^2=x_3^2=x_4^2$. 
Let $H=\langle x_1,x_2,x_3, x_4\rangle$. As usual, there is a homomorphism 
$\phi:G\rightarrow S_4$ given by $G$ acting on $\{H_1,H_2,H_3,H_4\}$ 
by conjugation (with $H_i$ identified with $i$).  
Since 
$x_i^2H_jx_i^{-2}=x_j^2H_jx_j^{-2}=H_j$, $\phi(x_i^2)=1$, 
so $\phi(x_i)$ must have order 1 or 2. 
Since $x_i H_i x_i^{-1}=H_i$, $\phi(x_i)$ must fix $i$. These two conditions 
force $\phi(x_i)$ to be the identity or a 2-cycle not containing $i$. 
Without loss of generality, we may assume that 
$\phi(x_1)=(1)$ or $\phi(x_1)=(34)$. In either case, 
$x_1 x_2 x_1^{-1}=x_2^{\pm 1}$. 
However, if 
$x_1 x_2 x_1^{-1}=x_2^{-1}$, 
then $1=x_1^6=x_2^2 x_1 x_2^2 x_1 = 
x_2^2 x_2^{-2}x_1x_1 = x_1^2$
contradicting the fact that $x_1$ has order 6. Therefore $x_1$ and $x_2$ 
commute and $\langle x_1,x_2\rangle\cong C_6\times C_2$, 
which has 3 cyclic subgroups 
of order 6.  Without loss of generality, we may assume that 
$x_3\in\langle x_1,x_2\rangle$. Since $x_1,x_2$, and $x_3$ all commute, 
$\phi(x_i)$ fixes 1,2, and 3 for $i=1,2,3$ and hence they must fix 4 as well. 
Consequently, $x_i x_4 x_i^{-1}=x_4^{\pm 1}$ for $i=1,2,3$. 
By an argument analogous to that above,  $x_i x_4 x_i^{-1}=x_4^{-1}$ 
leads to a contradiction, so $H$ must be an abelian group. 
But in that case, $\langle x_1^{-1} x_2 x_4\rangle$ is a fifth cyclic subgroup 
of order 6, yielding a contradiction and the result follows.
\end{proof}  

\begin{lemma}\label{lemma:445}
	There is no group $G$ such that $\sigma(G)=(4,4,5)$.
\end{lemma}

\begin{proof}
	There are elements $x, y, z\in G$ such that $\c{x}\cong C_5$, and 
	$\c{y}, \c{z}$ are the distinct cyclic subgroups of order $4$.
	When $G$ acts on $\{\c{y}, \c{z}\}$ by conjugation we have 
	a homomorphism
	$\phi:G\to S_2$, and since $x$ has order $5$ we must have 
	$\phi(x)=(1)$ in $S_2$.
	Now we have $\c{y}\unlhd\c{x, y}$, and $\c{x}\cap\c{y}=\{1\}$, 
	and hence 
	$\c{x, y} \cong C_4\semidirect C_5 \cong \c{y}\semidirect\c{x}$ for 
	some homomorphism $\c{x}\to\Aut(\c{y})$. 
	But the only homomorphism from $C_5$ to $\Aut(C_4)$ is 
	the trivial one.  Therefore the semidirect product
	is the direct product, and $G$ has a subgroup isomorphic to 
	$C_4\times C_5\cong C_{20}$ which is impossible since $G$ has no element
	of order $20$.
\end{proof}

\begin{lemma}\label{lemma:448}
There is no group $G$ such that $\sigma(G)=(4,4,8)$.
\end{lemma}

\begin{proof} Suppose there were such a group and let $K=\langle x\rangle$
	be the unique cyclic subgroup of order $8$.  
	Let $H_1=\langle y\rangle$ and $H_2=\langle x^2\rangle$ be the two 
	cyclic subgroups of order $4$.  Group $G$ acts on $\{H_1, H_2\}$
	by conjugation inducing a homomorphism $\phi:G\to S_2$.
	Because $xH_2x^{-1} = H_2$, we have $\phi(x)=(1)$. Therefore
	$xyx^{-1}=y$ or $xyx^{-1}=y^{-1}$.  In either case it will follow 
       that the order of $xy$ is $8$, a contradiction.

       \smallskip Case 1: $xy=yx$.  
       Since $(xy)^8=1$, the order of $xy$ divides $8$.  
       Since $x$ has order $8$ we have 
       $(xy)^4 = x^4y^4=x^4$ is not the identity. 
       Therefore, $xy$ has order $8$ in this case. 
       
       \smallskip Case 2: $xy=y^{-1}x$. 
       A computation shows $(xy)^4=x^4$ is not the identity, and $(xy)^8=1$.
       Therefore, $xy$ has order $8$ in this case as well.
       \end{proof}

\begin{prop}\label{prop:44444}
There is no group $G$ with $\sigma(G)=(4,4,4,4,4)$.
\end{prop}
\begin{proof}
	This follows from Proposition~\ref{odd_4s}.
\end{proof}

\begin{lemma}\label{lemma:33336}
There is no group $G$ with $\sigma(G)=(3,3,3,3,6)$.
\end{lemma}
\begin{proof}
Suppose $\c{x}\cong C_6$.
Choose an element $z\in G$ of order $3$ where $\c{x}\cap\c{z}=\{1\}$.
Since $\c{x}\unlhd G$, we have a semidirect product subgroup
$C_6\semidirect\c{z}$ for some homomorphism $C_3\to \Aut(C_6)\cong C_2$.
It follows that the homomorphism is trivial and 
$C_6\semidirect\c{z} \cong C_6\times C_3$.
Consequently $G$ contains more than one cyclic subgroup of order $6$,
a contradiction.
\end{proof}

\begin{lemma}\label{lemma:34466}
(a) There is no group $G$ with $\sigma(G)=(3,4,4,6,6)$.

(b) If $\sigma(G)=(3,4,4,4,6)$ then $G\cong C_3\semidirect C_4$.
\end{lemma}
\begin{proof}
Suppose $\c{x}\cong C_3$ and $\c{y}\cong C_4$.  These two groups intersect 
trivially, and the first is normal in $G$ so 
is normal in $\c{x,y}$, hence $\c{x,y}\cong C_3\semidirect C_4$.
If the action of $y$ on $\c{x}$ is trivial 
then $\c{x,y}\cong C_{12}$,
which is impossible if $\sigma(G)=(3,4,4,6,6)$ or
if $\sigma(G)=(3,4,4,4,6)$.
If the action is not trivial then
$$H=\c{x, y\ |\ x^3=y^4=1, \ yxy^{-1}=x^{-1} },$$
and in this case one can check that 
$\c{y}$, $\c{xy}$, and $\c{x^2y}$
are three distinct cyclic subgroups of order 4.
This is not
possible if $\sigma(G)=(3,4,4,6,6)$, so (a) is proved.

To prove (b), one readily checks that 
$\sigma(H)=(3,4,4,4,6)$, so by Proposition~\ref{from_H}, 
$G\cong C_3\semidirect C_4$. 
\end{proof}

We now prove Theorem~\ref{thm:diff5}.

\begin{proof}[Proof of Theorem \ref{thm:diff5}]
If $\Delta(G)=5$ then because of the previous Lemmas and Propositions
summarized in Revised Table \ref{table5.2}, either $\sigma(G)=(7)$ or $\sigma(G)=(3,4,4,4,6)$.
If $\sigma(G)=(7)$ then $G\cong C_7$ or $G\cong C_{14}$ by 
part (i) of Proposition~\ref{a_2a}.
If $\sigma(G)=(3,4,4,4,6)$ then $G\cong C_3\semidirect C_4$ by part (b) of
Lemma~\ref{lemma:34466}.
\end{proof}

\begin{table}[H]
	\caption{\label{table5.2} Revised table for $\Delta(G)=5$}
	\centering
\begin{tabular}{|c|c|c|}
\hline
Partition & $\sigma(G)$ & Groups \\
\hline
5 & (7) & $C_7$, $D_{14}$ \\
\hline
3+1+1 & (3,4,4,4,6) & $C_3\semidirect C_4$ \\
\hline
\end{tabular}
\end{table}

\begin{table}[H]
	\caption{\label{table5.1} Exclusion table for $\Delta(G)=5$}
	\centering
\begin{tabular}{|c|l|} 
	\hline
	$\sigma(G)$ & Excluded by \\
	\hline
	(9) & No $C_3$ \\
	(14) & No $C_7$ \\
	(18) & No $C_9$ \\
	(3,3,3,3,3) & Sylow \\
	(4,4,4,4,4) & Proposition \ref{prop:44444}\\
	(6,6,6,6,6) & No $C_3$ \\
	(3,3,3,3,4) & Lemma \ref{lemma:r} \\
	(3,3,3,3,6) & Lemma \ref{lemma:33336} \\
	(3,4,4,4,4) & Lemma \ref{lemma:les1} \\
	(4,4,4,4,6) & No $C_3$ \\
	(3,6,6,6,6) & Lemma \ref{lemma:36666}\\
	(4,6,6,6,6) & No $C_3$ \\
	(3,3,3,4,4), (3,3,3,6,6),(3,3,4,4,4)  & Sylow \\
	(4,4,4,6,6) & No $C_3$ \\
	(3,3,6,6,6) & Sylow \\
	(4,4,6,6,6) & No $C_3$ \\
	(3,3,5) & Sylow \\
	(4,4,5) & Lemma \ref{lemma:445} \\
	(5,6,6) & No $C_3$ \\
	(3,3,8) & No $C_4$ \\
	(4,4,8) & Lemma \ref{lemma:448} \\
	(6,6,8) & No $C_3$ \\
	(3,3,10), (4,4,10), (6,6,10)  & No $C_5$ \\
	(3,3,12) & No $C_4$ \\
	(4,4,12), (6,6,12) & No $C_3$ \\
	(3,3,3,4,6) & Sylow \\
	(3,4,6,6,6) & Proposition \ref{lemma:ab} \\
	(3,4,5), (3,5,6) & Proposition \ref{lemma:ab} \\
	(4,5,6) & No $C_3$ \\
	(3,4,8), (3,6,8) & Proposition \ref{lemma:ab} \\
	(4,6,8) & No $C_3$ \\
	(3,4,10), (3,6,10), (4,6,10)  & No $C_5$ \\
	(3,4,12) & No $C_6$ \\
	(3,6,12) & No $C_4$ \\
	(4,6,12) & No $C_3$ \\
	(3,3,4,4,6), (3,3,4,6,6)  & Sylow \\
	(3,4,4,6,6) & Lemma \ref{lemma:34466} \\
	\hline
\end{tabular}
\end{table}

\section{Acknowledgements}
This paper is an expansion and revision of the Master's thesis of the 
second author, which was directed by the first and third authors.

We would like to thank the anonymous referee for a very careful reading 
of the paper, and for many helpful remarks.

\end{document}